\documentclass[10pt,notitlepage,a4paper,reqno,ps]{amsart}
\usepackage{eurosym}
\usepackage{amssymb}
\usepackage{amstext}
\usepackage{color}
\usepackage{enumitem}
\usepackage{esint}
\usepackage{graphicx} 

\newcommand{\N}{\mathbb{N}}
\newcommand{\R}{\mathbb{R}}

\newtheorem{thm}{Theorem}[section]
\newtheorem{prop}[thm]{Proposition}
\newtheorem{lemma}[thm]{Lemma}

\newtheorem{definition}[thm]{Definition}

\title[Local Lipschitz continuity with lower order terms]{Local Lipschitz continuity for energy integrals with fast
	growth and lower order terms}

\date{}

\author[A. Torricelli]{Andrea Torricelli}
\address{Andrea Torricelli\\
	Dipartimento di Scienze Matematiche G.L.\,Lagrange, Politecnico di Torino, C.so Duca degli Abruzzi 24, 10129 Torino, Italy
}
\email{andrea.torricelli@polito.it}

\begin{document}
	\maketitle
	{\bf Abstract.} We consider integral functionals with fast growth and the lagrangian explicitly depending on $u$. We prove that the local minimizers are locally Lipschitz continuous.

	\section{Introduction}
	Lipschitz regularity for local minimizers of integral functionals and weak solutions to nonlinear
	elliptic partial differential equations in divergence form with non-standard growth conditions is a wide and active field of research, see for instance some recent contributions \cite{BM, BS2022, CMMPdN, DFM2021, EMM, EMMP, EPdN, EPT, GGT, M2021}. The results proved in the present paper are a contribution to this field. Indeed, our paper aims to prove local Lipschitz regularity results for integral functionals of the type
	\begin{equation*}
		\label{funz-modello}
		\mathcal{F}(u)=\int_{\Omega} f(Du) + g(x, u) \, dx.
	\end{equation*}
	The interest in kind of functionals is not purely academical. Indeed, these functionals have many applications, such as in elastoplastic torsion problems or in image restoration problems (see \cite{GT}). Furthermore, this class of functionals is well known in literature, see for instance: \cite{CCG, DFM2021} concerning
	regularity of local minimizers of a class of integrals of the Calculus of Variations, \cite{DFM2022} where the functionals considered do not necessarily satisfy the Euler–Lagrange equation, and \cite{EPT} concerning the local Lipschitz continuity of local minimizers for functionals with slow growth.\\
	
	\noindent
	Our work was inspired by \cite{EMMP, Marc96}, where the authors deal with functionals depending only on $Du$  with general growth assumptions, respectively fast and slow, and by \cite{EPT} where the authors deal with functionals depending on both $u$ and $Du$, but only in the case of slow growth. Indeed, our aim is to apply the techniques used in \cite{EPT} in order to obtain an analogous result but for functionals with fast growth. In the aforementioned papers the authors first prove a suitable a priori estimates, and then apply classical results on the Bounded Slope Condition to get the local Lipschitz continuity. In order to follow the same steps we need to exploit some recent results (see \cite{FT, GT}) that extend the classical ones on the Bounded Slope Conditions (BSC) to our type of functionals.\\

	\noindent
	Our aim is to prove the local Lipschitz continuity of the minimizers of the functional
	\begin{equation}
		\label{functional}
		\int_\Omega f(D u)+ g(x,u)dx, \quad u\in W^{1,1}_{loc}(\Omega),
	\end{equation}
	with $\Omega\subset\R^n$ open and bounded, $n\in\N$ with $n\ge 2$, $f:\R^n\to \R$ a non-negative, convex function of class $C(\R^n)\cap C^2(\R^n\setminus B_{t_0}(0)$ for some $t_0\ge 0$, $g:\R^n\times\R\to\R$ Carathéodory,  and $h_1,h_2:[0,\infty)\to[0,\infty)$ non-identically zero and increasing functions. We assume the following conditions hold on $f$ for all $\xi,\lambda\in\R^n$ such that $|\xi|\ge t_0$:
	\begin{equation}
		\label{H1}
		h_1(|\xi|)|\lambda|^2\le\sum_{i,j=1}^n f_{\xi_i, \xi_j}(\xi)\lambda_1\lambda_j\le h_2(|\xi|)|\lambda|^2; \tag{H1}\
	\end{equation}
	\begin{equation}
		\label{H2}
		h_2(t)t^2\le c_1\left[1+\int_0^t\sqrt{h_1(s+t_0)}ds\right]^\alpha \text{ for $t\ge t_0$}; \tag{H2}
	\end{equation}
	\begin{equation}
		\label{H3}
		h_2(|\xi|)|\xi|^2\le c_2(1+f(\xi))^\beta
		\tag{H3},
	\end{equation}
	for some positive constants $c_1$ and $c_2$, and for $\alpha, \beta$ such that
	\begin{equation*}
		2\le\alpha<2^*, \quad 1\le \beta< \frac{2}{n}\frac{\alpha}{\alpha-2},
	\end{equation*}
	where $f_{\xi_i, \xi_j}$ denotes the $i,j$-th component of the Hessian matrix of $f$ and $2^*:=\frac{2n}{n-2}$ denotes the Sobolev conjugate of $2$. Moreover, we assume the following conditions on $g$ for a.e. $x\in\Omega$:
	\begin{equation}
		\label{G1}|g(x.s_1)-g(x,s_2)|\le L|s_1-s_2|\text{ and for all }s_1,s_2\in\R;\tag{G1}
	\end{equation}
	\begin{equation}
		\label{G2}g(\cdot, 0)\in L^1_{loc}(\Omega);\tag{G2}
	\end{equation}
	\begin{equation}
		\label{G3}s\mapsto g(x,s) \text{ convex},\tag{G3}
	\end{equation}
	for some positive constant $L$. Finally we assume that there exists $K>0$ such that for a.e. $x,y \in \Omega$ it holds
	\begin{equation}
		\label{G4}v\ge u + K|y-x| \Rightarrow g_v^+(y,v)\ge g_v^+(x,u),\tag{G4}
	\end{equation}
	where $g_v^+$ denotes the right derivative of $g$ with respect to the second variable.
	Our main result is the following.
	\begin{thm}
		\label{Lipschitz}
		Given $u\in W_{loc}^{1,1}(\Omega)\cap L_{loc}^\infty(\Omega)$ local minimizer (see Definition \ref{local_min}) of the functional \eqref{functional}. Assume that $f$ satisfies the growth assumptions \eqref{H1}-\eqref{H3}, and that $g$ satisfies hypothesis \textnormal{(G1)}--\textnormal{(G4)}.
		\\
		Then $u$  is locally Lipschitz continuous in $\Omega $ and  there exists $\bar R>0$ such that for every $0<\rho< R<\bar R$,  there exist $\theta\in \R$ and two positive constants $C$ and $\kappa,$ depending on the data of the problem, with $
		C $ depending also on $\rho $, $R$  
		while $\kappa$ depends also on $\|u\|_{L^{\infty}(B_R)}$, $\|g(\cdot, 0)\|_{L^1(B_R)}$ and $|B_R|,$ 
		such that 
		\begin{equation*}
			\Vert Du\Vert _{L^{\infty }(B_{\rho }\,;\mathbb{R}^{n})}\leq \,C\,\left[ \frac{1}{(R-\rho )^{n}} \left (\int_{B_{R}} f(Du)+g(x,u)\,dx +\kappa \right )\right] ^{\theta }
			\label{infinity nu}
		\end{equation*}%
		with $\theta$ depending on $\alpha,\beta,$ and $n$.
	\end{thm}
	\noindent
	Fundamental to prove Theorem \ref{Lipschitz} will be the following apriori estimate.
	\begin{lemma}
		\label{apriori_estimate}
		Let $f$ satisfy hypothesis \eqref{H1} to \eqref{H3}, $g$ satisfy hypothesis \eqref{G1}, and $u\in W^{1,\infty}_{loc}(\Omega)$ be a local minimizer of \eqref{functional}.
		Then, for every $R>0$ small enough and $\rho$ such that $0<\rho<R$, there exists $\theta\in \R$ and a positive constant $C$ depending on $\alpha, \beta, R, \rho, c_1,c_2, h_1(t_0)$, and $n$ such that
		\begin{equation}
			\label{apriori}
			||D u||_{L^\infty(B_\rho, \R^n)}\le C\left[\int_{B_{R}}1+f(Du)dx\right]^\theta
		\end{equation}
		with $\theta$ depending on $\alpha,\beta,$ and $n$.
	\end{lemma}
	\section{Preliminaries and Notation}
	Given $f:\R^n\rightarrow[0,+\infty)$ convex, we denote with $\partial f(x)$ the subdifferential of $f$ at the point $x$.
	We indicate by $f^*$ the polar function, or Fenchel tranform, of $f$ (see \cite{RW}), defined as
	$$f^*(x):=\sup_{\xi\in \R^n}\{x\cdot\xi-f(\xi)\}, \quad \forall x\in \R^n.$$ If $f$ is differentiable, then $f_{\xi_i}$ denotes the $i$-th component of $Df$. If $f$ is two times idfferentiable, then $f_{\xi_i,\xi_j}$ denotes the $i,j$-component of the Hessian matrix of $f$.
	We denote by $B_R$ a $n$-dimensional ball of radius $R>0$, the center will usually be omitted unless significant at the moment. Given an open and bounded set $S\subset\R^n$, we denote by $W^{1,1}_0(S)$ the set of $W^{1,1}(S)$-functions with null trace on $\partial S$.
	\begin{definition}[Local minimizer]
		\label{local_min}
		We say that $u\in W_{loc}^{1,1}(\Omega )$ is a \textit{local minimizer} of \eqref{functional} if $f(Du) + g(\cdot,u) \in
		L_{loc}^{1}(\Omega )$ and 
		\begin{equation*}
			\int_{\Omega ^{\prime }}f(Du)+g(x,u)\,dx\leq \int_{\Omega ^{\prime }}f(Du+D\varphi
			)+g(x,u+\varphi)\,dx
		\end{equation*}%
		for every open set $\Omega ^{\prime }\Subset
		\Omega $ and for every $\varphi \in W_{0}^{1,1}(\Omega ^{\prime })$.
	\end{definition}
	
	\begin{definition}[BSC]
		\label{bscdefn}
		A function $\phi:\Omega\to\R$ satisfies the \emph{Bounded Slope Condition} of
		rank $m\ge 0$ if for every $\gamma\in\partial\Omega$ there exist
		$z_{\gamma}^-$, $z_{\gamma}^+\in \R^n$ and $m\in\R$ such that
		\begin{equation*}
			\phi(\gamma)+z_{\gamma}^-\cdot(\gamma'-\gamma)\le\phi(\gamma')\le \phi(\gamma)+z_{\gamma}^+\cdot(\gamma'-\gamma)
		\end{equation*}
		and $|z_{\gamma}^{\pm}|\le m$ for every $\gamma\in\partial\Omega$.
	\end{definition}

	\begin{prop}
		\label{polare}
		Let $f:\R^n\rightarrow[0,+\infty)$ be a convex function such that $f\in\mathcal{C}(\mathbb{R}^{n})\cap \mathcal{C}^{2}(\mathbb{R%
		}^{n}\backslash B_{t_{0}}(0))$ for some $t_{0}\geq 0$, superlinear, and satisfies assumption {\rm (F1)} then there exists $s_0\in\R$, such that $f^*\in\mathcal{C}^2(\R^n\setminus B_{s_0}(0))$. Furthermore, given $x\in\R^n\setminus B_{s_0}(0)$, it holds $D^2f^*(x)=\left(D^2f\right)^{-1}(Df^*(x))$ and 
		\begin{equation*}
			\displaystyle\frac{1}{h_{2}\left( \left\vert Df^*(x) \right\vert \right) }\left\vert \lambda
			\right\vert ^{2}\leq \sum_{i,j=1}^{n}f^*_{x _{i}x _{j}}\left( x \right)
			\lambda _{i}\lambda _{j}\leq \frac{1}{h_{1}\left( \left\vert Df^*(x) \right\vert \right)}
			\left\vert \lambda \right\vert ^{2},    
		\end{equation*}
		for every $\lambda ,x \in 
		\mathbb{R}^{n}$ such that $\left\vert x \right\vert {\geq }s_{0}$.
	\end{prop}
	\begin{proof}
		It follows from \cite[Theorem 13.21]{RW} and the subsequent observation.     
	\end{proof}
	
	\begin{lemma}[\cite{EPT}, Lemma 2.4]
		\label{Lemmabound}
		Assume that $f_-,\tilde f,f_+:\R^n\rightarrow\R$ are convex and superlinear. Moreover, assume that 
		\begin{equation*}\label{fmenopiù}
			f_-(\xi)\le \tilde f(\xi)\le f_+(\xi).
		\end{equation*}
		Let $g:\Omega\times\R\rightarrow\R$  a Carathéodory function satisfying \textnormal{(G1)}. 
		Let $\varphi$ satisfy the  
		{\rm (BSC)} and $\tilde u$ be a minimizer of 
		\begin{equation*}\label{funzftilde}
			\int_{B_R} \tilde f(D v)+g(x,v)\, dx \qquad v\in\varphi+W^{1,1}_0(B_R).
		\end{equation*}
		Then there exists $\bar K=\bar K(L,f_-,f_+,\|\varphi\|_{L^\infty(B_R)})$
		such that $\|\tilde u\|_{L^{\infty}(B_R)}\le \bar K$. 
	\end{lemma}
	
	\section{Proof of Lemma \ref{apriori_estimate}}
	Since $u$ is a local minimizer of \eqref{functional}, then on every $\Omega'\Subset \Omega$ it satisfies the Euler-Lagrange equation 
	\begin{equation*}
		\int_\Omega \sum_{i=1}^n f_{\xi_i}(D u)\phi_{x_i}+g_v(x,u)\phi dx=0, \text{ for every }\phi \in W_0^{1,2}(\Omega'),
	\end{equation*}
	where $\phi_{x_i}$ denotes the $i$-th component of $\nabla \phi$. By difference quotients techniques (see \cite[Theorem 1.1', Ch II]{Gia83}), then $u\in W_{loc}^{2,2}(\Omega)$ and the second variation
	\begin{equation}
		\label{EL2}
		\int_\Omega\sum_{i,j=1}^n f_{\xi_i,\xi,j}(D u)u_{x_j,x_k}\phi_{\xi_i}-g_v(x,u)\phi_{x_k}dx=0, \quad k\in\{1,...,n\}
	\end{equation}
	holds for every $\phi\in W_0^{1,2}(\Omega')$. Fix $k\in \{1,...,n\}$ and let $\eta\in C^1_0(\Omega')$ such that $\eta=1$ on $B_\rho$, ${\rm supp}(\eta)\subset B_R$, and $|D \eta|\le \frac{2}{R-\rho}$. Moreover, let $\Phi:[0,\infty)\to[0,\infty)$ be a non-negative, non-decreasing, locally Lipschitz function such that $\Phi(0)=0$. Furthermore, for every $a\in\R$ let $(a)_+:=\max\{a,0\}$. Denote $\varphi:=\eta^2u_{x_k}\Phi(|D u|-1)_+$, where for notation simplicity
	\begin{equation*}
		\Phi(|D u|-1)_+:=\Phi((|D u|-1)_+).
	\end{equation*}
	By construction, $\varphi\in C_0^1(\Omega')$ and 
	\begin{equation*}
		\varphi_{\xi_i}=2\eta\eta_{x_i}u_{x_k}\Phi(|D u|-1)_+ + \eta^2u_{x_k}\Phi'(|D u|-1)_+[(|D u|-1)_+]_{x_i}.
	\end{equation*}
	Plugging $\varphi$ in \eqref{EL2} we get
	\begin{align*}
		&\int_{\Omega }2\eta \Phi (|Du|-1)_{+} \sum_{i,j=1}^{n} \eta _{x_{i}}u_{x_{k}} f_{\xi _{i}\xi
			_{j}}(Du)u_{x_{j}x_{k}} \,dx \\
		& -\int_{\Omega }2\eta \Phi (|Du|-1)_{+}  \eta _{x_{k}}u_{x_{k}} g_u(x,u) \,dx \\
		&+\int_{\Omega }\eta ^{2}\Phi (|Du|-1)_{+}\sum_{i,j=1}^{n}f_{\xi _{i}\xi
			_{j}}(Du)u_{x_{j}x_{k}}u_{x_{i}x_{k}}\,dx \\
		&- \int_{\Omega }\eta^2 \Phi (|Du|-1)_{+} u_{x_{k} x_k} g_u(x,u) \,dx \\
		&+\int_{\Omega }\eta ^{2}\Phi ^{\prime }(|Du|-1)_{+}\sum_{i,j=1}^{n}f_{\xi
			_{i}\xi _{j}}(Du)u_{x_{j}x_{k}}u_{x_{k}}[(|Du|-1)_{+}]_{x_{i}}\,dx \\
		&-\int_{\Omega }\eta ^{2}\Phi ^{\prime }(|Du|-1)_{+} g_u(x,u) u_{x_{k}}[(|Du|-1)_{+}]_{x_{k}}\,dx \\
		&= I_{1k} - I_{2k} + I_{3k} - I_{4k} + I_{5k} - I_{6k}=0.
	\end{align*}
	Summing for $k\in\{1,...,n\}$, we get
	\begin{equation}\label{Eq delle I}
		I_3+I_5\le |I_1|+|I_2|+|I_4|+|I_6|.
	\end{equation}
	Now we estimate each term separately. We start from $I_1$, for which we use, in order, Cauchy-Schwartz inequality, Young inequality, and \eqref{H1}.
	\begin{align}
		\label{I_1}
		&|I_{1}| = \left | \int_{\Omega }2\eta \Phi (|Du|-1)_{+}\sum_{i,j, k=1}^{n}f_{\xi
			_{i}\xi _{j}}(Du)u_{x_{j}x_{k}}\eta _{x_{i}}u_{x_{k}}\,dx\right| \nonumber\\
		&\leq \int_{\Omega }2\Phi (|Du|-1)_{+}\left( \eta
		^{2}\sum_{i,j, k=1}^{n}f_{\xi _{i}\xi
			_{j}}(Du)u_{x_{i}x_{k}}u_{x_{j}x_{k}}\right) ^{\frac{1}{2}}\nonumber \\
		& \times \left(
		\sum_{i,j, k=1}^{n}f_{\xi _{i}\xi _{j}}(Du)\eta _{x_{i}}u_{x_{k}}\eta
		_{x_{j}}u_{x_{k}}\right) ^{\frac{1}{2}}\,dx \nonumber\\
		& \leq \frac{1}{2}\int_{\Omega }\eta ^{2}\Phi
		(|Du|-1)_{+}\sum_{i,j, k=1}^{n}f_{\xi _{i}\xi
			_{j}}(Du)u_{x_{i}x_{k}}u_{x_{j}x_{k}}\,dx \nonumber \\
		& +2\int_{\Omega }\Phi (|Du|-1)_{+}\sum_{i,j, k=1}^{n}f_{\xi _{i}\xi
			_{j}}(Du)\eta _{x_{i}}u_{x_{k}}\eta _{x_{j}}u_{x_{k}}\,dx\nonumber\\
		&\leq \frac{1}{2}\int_{\Omega }\eta ^{2}\Phi
		(|Du|-1)_{+}\sum_{i,j, k=1}^{n}f_{\xi _{i}\xi
			_{j}}(Du)u_{x_{i}x_{k}}u_{x_{j}x_{k}}\,dx \nonumber\\
		&+2\int_\Omega \eta ^{2}\Phi
		(|Du|-1)_{+} \sum_{k=1}^n u_{x_k}^2 h_2(|D u|)|D \eta|^2dx\nonumber\\
		&\leq \frac{1}{2}\int_{\Omega }\eta ^{2}\Phi
		(|Du|-1)_{+}\sum_{i,j, k=1}^{n}f_{\xi _{i}\xi
			_{j}}(Du)u_{x_{i}x_{k}}u_{x_{j}x_{k}}\,dx \nonumber\\
		&2\int_\Omega \eta ^{2}\Phi
		(|Du|-1)_{+} |D \eta|^2 h_2(|D u|)|D u|^2dx\nonumber\\
		&\leq \frac{1}{2}\int_{\Omega }\eta ^{2}\Phi
		(|Du|-1)_{+}\sum_{i,j, k=1}^{n}f_{\xi _{i}\xi
			_{j}}(Du)u_{x_{i}x_{k}}u_{x_{j}x_{k}}\,dx \nonumber\\
		&2\int_\Omega \eta ^{2}\Phi
		(|Du|-1)_{+} |D \eta|^2 h_2(|D u|)|D u|^2dx\nonumber\\
		&\leq \frac{1}{2}\int_{\Omega }\eta ^{2}\Phi
		(|Du|-1)_{+}\sum_{i,j, k=1}^{n}f_{\xi _{i}\xi
			_{j}}(Du)u_{x_{i}x_{k}}u_{x_{j}x_{k}}\,dx \nonumber\\
		&2\int_\Omega \eta ^{2}\Phi
		(|Du|-1)_{+} |D \eta|^2 h_2(1+(|D u|-1)_+)(1+(|D u|-1)_+)^2dx
	\end{align}
	For $I_2$ we use, in order, \eqref{G1}, Young inequality, and the fact $h_2(t)\ge h_1(t)\ge h_1(1)/t$ for $t\ge 1$.
	\begin{align}\label{I_2}
		|I_{2}| & = \left\vert \int_{\Omega }  2\eta \Phi (|Du|-1)_{+} \sum_{k=1}^n \eta _{x_{k}}u_{x_{k}} g_u(x,u) \,dx
		\right\vert\nonumber\\
		& \le 2L\int_\Omega |Du||D\eta||\eta|\Phi (|Du|-1)_{+}dx\nonumber\\
		&\leq \, L  \int_{\Omega} (\eta^2 + |D \eta|^2) |Du| \,  \Phi (|Du|-1)_{+} \, dx\nonumber\\
		& \leq \, \frac{L}{h_1(1)} \int_{\Omega} (\eta^2 + |D \eta|^2) \Phi (|Du|-1)_{+} h_2(|Du|) \, |Du|^2 \, dx\nonumber\\
		&\le \frac{L}{h_1(1)} \int_{\Omega} (\eta^2 + |D \eta|^2) \Phi (|Du|-1)_{+} h_2(1 + (|Du|-1)_+) \, (1 + (|Du|-1)_+)^2  \, dx.
	\end{align}

	For $|I_4|$ we have
	\begin{align}
		\label{I_4}
		|I_4| & = \left\vert \int_{\Omega }  \eta^2 \Phi (|Du|-1)_{+} \sum_{k=1}^n u_{x_{k} x_k}  g_u(x,u) \,dx
		\right\vert \nonumber \\
		& \le nL\int_\Omega \eta^2  \Phi (|Du|-1)_{+} |D^2|^2dx \nonumber\\
		& \leq \, \frac{nL}{h_1(1)} \int_{\Omega} \left [\eta^2 |D^2u|^{2}  \Phi (|Du|-1)_{+} h_1(1 + (|Du| - 1)_+)\right ]^{\frac{1}{2}} \nonumber \\
		& \quad \times \left [\eta^2 \Phi (|Du|-1)_{+} h_2(1 + (|Du| - 1)_+) (1 + (|Du|-1)_+))^2\right ]^{\frac{1}{2}} \, dx \nonumber\\
		&  \leq \, \varepsilon \, \int_{\Omega} \eta^2  \Phi (|Du|-1)_{+} h_1(1 + (|Du| - 1)_+) |D^2 u|^{2} \, dx \nonumber\\
		& \quad + \frac{n^2 L^2}{4 h_1(1)^2 \varepsilon}  \int_{\Omega} \eta^2 \Phi (|Du|-1)_{+} h_2(1 + (|Du| - 1)_+) (1 + (|Du|-1)_+)^2 \, dx.
	\end{align}
	Where we used the fact that 
	\begin{equation}
		\label{g1g2}
		h_1(t) \, h_2(t) \, t^2\ge h_1(t)^2t^2 \ge h_1(1)^2, \quad \text{ for } t\ge 1.
	\end{equation}
	Before estimating $I_5$, we observe that
	\begin{equation*}
		[(|Du|-1)_+]_{x_i}=
		\begin{cases}
			\frac{1}{|Du|}\sum_{k=1}^n u_{x_i,x_k} u_{x_k} \quad \textbf{if } |Du|>1\\
			0 \qquad \text{if }|Du|\le 1
		\end{cases}
	\end{equation*}
	so for $I_5$ we obtain
	\begin{align}
		\label{I_5}
		I_5 = & \int_{\Omega} \eta^2 \Phi'(|Du|-1)_{+} \sum_{i,j,k=1}^{n} f_{\xi _{i}\xi
			_{j}}(Du)u_{x_{j}x_{k}}u_{x_{k}}[(|Du|-1)_{+}]_{x_{i}} \, dx\nonumber \\
		= & \int_{\Omega} \eta^2 \Phi'(|Du|-1)_{+} |Du|%
		\sum_{i,j=1}^{n}f_{\xi _{i}\xi
			_{j}}(Du)[(|Du-1|)_{+}]_{x_{j}}[(|Du|-1)_{+}]_{x_{i}} \, dx.
	\end{align}
	
	Finally, we estimate $|I_6|$ as
	\begin{align}
		\label{I_6}
		|I_{6}| & = \left\vert \int_{\Omega }\eta ^{2}\Phi ^{\prime }(|Du|-1)_{+} g_u(x,u) \sum_{k=1}^n u_{x_{k}}[(|Du|-1)_{+}]_{x_{k}}\,dx
		\right\vert \nonumber\\
		& \leq  \, L \int_{\Omega} \eta ^{2}\Phi ^{\prime }(|Du|-1)_{+} (1 + (|Du|-1)_+) |D(|Du|-1)_{+}|\,dx\nonumber\\
		& {\leq} \, \frac{L}{h_1(1)} \int_{\Omega} \eta ^{2}\Phi ^{\prime }(|Du|-1)_{+} (1 + (|Du|-1)_+) |D(|Du|-1)_{+}|\nonumber \\ 
		& \quad \times \sqrt{h_1(1 + (|Du|-1)_+) \, h_2(1 + (|Du| -1)_+)} (1 + (|Du| - 1)_+)\,dx\nonumber\\
		& \leq  \, \frac{L}{h_1(1)} \int_{\Omega} \Big [\eta^2 |D(|Du|-1)_{+}|^{2} \Phi ^{\prime }(|Du|-1)_{+} (1 + (|Du|-1)_+)^2\nonumber\\
		& \quad \times h_1(1 + (|Du| - 1)_+)\Big ]^{\frac{1}{2}} \Big [\eta^2 \Phi ^{\prime }(|Du|-1)_{+} \nonumber \\
		& \times h_2(1 + (|Du| - 1)_+) (1 + (|Du|-1)_+)^2\Big ]^{\frac{1}{2}} \, dx\nonumber\\
		&  \leq  \varepsilon \, \int_{\Omega} \eta^2  \Phi ^{\prime }(|Du|-1)_{+} (1 + (|Du|-1)_+)^2 h_1(1 + (|Du| - 1)_+) |D(|Du|-1)_{+}|^{2} \, dx \nonumber\\
		& \quad + \, \frac{L^2}{4 h(1)^2 \varepsilon} \int_{\Omega} \eta^2 \Phi ^{\prime }(|Du|-1)_{+} h_2(1 + (|Du| - 1)_+) (1 + (|Du|-1)_+)^2 \, dx, 
	\end{align}
	where in the second inequality we used \eqref{g1g2}. Plugging \eqref{I_1},\eqref{I_2},\eqref{I_4},\eqref{I_5}, and \eqref{I_6} in \eqref{Eq delle I} we get
	\begin{align*}
		& \int_{\Omega }\eta ^{2}\Phi (|Du|-1)_{+} \sum_{i,j,k=1}^{n} f_{\xi _{i}\xi
			_{j}}(Du)u_{x_{j}x_{k}}u_{x_{i}x_{k}}\,dx \nonumber \\
		& + \int_{\Omega} \eta^2 \Phi'(|Du|-1)_{+} |Du| \sum_{i,j =1}^{n} f_{\xi _{i}\xi
			_{j}}(Du)[(|Du-1|)_{+}]_{x_{j}}[(|Du|-1)_{+}]_{x_{i}} \, dx \nonumber\\
		\le & \frac{1}{2}\int_{\Omega }\eta ^{2}\Phi
		(|Du|-1)_{+}\sum_{i,j, k=1}^{n}f_{\xi _{i}\xi
			_{j}}(Du)u_{x_{i}x_{k}}u_{x_{j}x_{k}}\,dx \nonumber \\
		& + 2 \int_{\Omega } |D \eta|^2 \, \Phi (|Du|-1)_{+} h_2(1 + (|Du|-1)_+) \, (1 + (|Du|-1)_+)^2 \,dx \nonumber\\
		& + \frac{L}{h_1(1)} \int_{\Omega} (\eta^2 + |D \eta|^2) \Phi (|Du|-1)_{+} h_2(1 + (|Du|-1)_+) \, (1 + (|Du|-1)_+)^2  \, dx \nonumber\\
		& + \varepsilon \, \int_{\Omega} \eta^2  \Phi (|Du|-1)_{+} h_1(1 + (|Du| - 1)_+) |D^2 u|^{2} \, dx \nonumber\\
		&+ \frac{n^2L^2}{4 h(1)^2 \varepsilon}   \int_{\Omega} \eta^2 \Phi (|Du|-1)_{+} h_2(1 + (|Du| - 1)_+) (1 + (|Du|-1)_+)^2 \, dx \nonumber \\
		& +  \varepsilon \, \int_{\Omega} \eta^2  \Phi ^{\prime }(|Du|-1)_{+} (1 + (|Du|-1)_+) h_1(1 + (|Du| - 1)_+) |D(|Du|-1)_{+}|^{2} \, dx \nonumber \\
		& + \frac{L^2}{4 h(1)^2 \varepsilon}  \int_{\Omega} \eta^2 \Phi ^{\prime }(|Du|-1)_{+} h_2(1 + (|Du| - 1)_+) (1 + (|Du|-1)_+)^2 \, dx.
	\end{align*}
	Now we absorb the first term in the right-hand side into the left hand-side of the inequality and we obtain
	\begin{align}
		\label{asterisco}
		& \frac{1}{2}\int_{\Omega }\eta ^{2}\Phi (|Du|-1)_{+} \sum_{i,j,k=1}^{n} f_{\xi _{i}\xi
			_{j}}(Du)u_{x_{j}x_{k}}u_{x_{i}x_{k}}\,dx \nonumber \\
		& + \int_{\Omega} \eta^2 \Phi'(|Du|-1)_{+} |Du| \sum_{i,j =1}^{n} f_{\xi _{i}\xi
			_{j}}(Du)[(|Du-1|)_{+}]_{x_{j}}[(|Du|-1)_{+}]_{x_{i}} \, dx \nonumber\\
		& \le  \varepsilon \, \int_{\Omega} \eta^2  \Phi ^{\prime }(|Du|-1)_{+} (1 + (|Du|-1)_+) h_1(1 + (|Du| - 1)_+) |D(|Du|-1)_{+}|^{2} \, dx \nonumber \\
		& + \varepsilon \, \int_{\Omega} \eta^2  \Phi (|Du|-1)_{+} h_1(1 + (|Du| - 1)_+) |D^2 u|^{2} \, dx \nonumber\\
		& + \frac{C}{\varepsilon} \int_{\Omega} (\eta^2 + |D \eta|^2) \Phi (|Du|-1)_{+} h_2(1 + (|Du|-1)_+) \, (1 + (|Du|-1)_+)^2  \, dx \nonumber\\
		&+ \frac{C}{\varepsilon}   \int_{\Omega} \eta^2 \Phi' (|Du|-1)_{+} h_2(1 + (|Du| - 1)_+) (1 + (|Du|-1)_+)^2 \, dx,
	\end{align}
	where $C$ is a constat depending on $n, L$, and $h_1(1)$. In what follows the constant $C$ may change from line to line, but it will always depend only on $n, L$, and $h_1(1)$ while $c$ will denote a generic positive constant. Using hypothesis \eqref{H1} and the inequality
	\begin{equation*}
		|D(|Du|-1)_+|^2\le |D^2 u|^2,
	\end{equation*}
	we get 
	\begin{align}
		\label{disug che non capivo}
		& \frac{1}{2}\int_{\Omega }\eta ^{2}\Phi (|Du|-1)_{+} \sum_{i,j,k=1}^{n} f_{\xi _{i}\xi
			_{j}}(Du)u_{x_{j}x_{k}}u_{x_{i}x_{k}}\,dx \nonumber \\
		& + \int_{\Omega} \eta^2 \Phi'(|Du|-1)_{+} |Du| \sum_{i,j =1}^{n} f_{\xi _{i}\xi
			_{j}}(Du)[(|Du-1|)_{+}]_{x_{j}}[(|Du|-1)_{+}]_{x_{i}} \, dx \nonumber\\
		& \ge c \int_\Omega \eta^2 \left(\Phi (|Du|-1)_{+} + |Du|\Phi' (|Du|-1)_{+}\right) h_1((|Du|-1)_{+})|D(|Du|-1)_+|^2dx .    \end{align}
	
	Assuming $\varepsilon$ small enough and combining \eqref{disug che non capivo} with \eqref{asterisco} we get
	
	\begin{align}
		\label{si usa?2}
		& \int_\Omega\eta^2 \left(\Phi (|Du|-1)_{+} + |Du|\Phi' (|Du|-1)_{+}\right) h_1((|Du|-1)_{+})|D(|Du|-1)_+|^2dx\nonumber\\
		& \le C\int_\Omega (\eta^2+ |D\eta|^2) \Phi (|Du|-1)_{+} h_2(1+(|Du|-1)_{+})(1 + (|Du|-1)_+)^2 dx\nonumber\\
		& + C\int_\Omega (\eta^2 +  |D\eta|^2) \Phi' (|Du|-1)_{+} h_2(1+(|Du|-1)_{+})(1 + (|Du|-1)_+)^2 dx,
	\end{align}
	where we absorbed the first two terms of the right-hand side of  \eqref{asterisco} into the left-hand side. Moreover, we can further estimate the left-hand side of the previous inequality with
	\begin{align*}
		& \int_\Omega\eta^2 \left(\Phi (|Du|-1)_{+} + |Du|\Phi' (|Du|-1)_{+}\right) h_1(1+(|Du|-1)_{+})|D(|Du|-1)_+|^2dx\\
		& \ge \int_\Omega\eta^2 \Phi (|Du|-1)_{+} h_1(1+(|Du|-1)_{+})|D(|Du|-1)_+|^2dx,
	\end{align*}
	which leads to 
	\begin{align}
		\label{si usa?}
		&\int_\Omega\eta^2 \Phi (|Du|-1)_{+} h_1(1+(|Du|-1)_{+})|D(|Du|-1)_+|^2dx\nonumber\\
		& \le C\int_\Omega (\eta^2+ |D\eta|^2) \Phi (|Du|-1)_{+} h_2(1+(|Du|-1)_{+})(1 + (|Du|-1)_+)^2 dx\nonumber\\
		& + C\int_\Omega (\eta^2 +  |D\eta|^2) \Phi' (|Du|-1)_{+} h_2(1+(|Du|-1)_{+})(1 + (|Du|-1)_+)^2 dx.
	\end{align}
	Now we denote
	\begin{equation*}
		G(t)=1+\int_0^t\sqrt{\Phi(s)h_1(1+s)}ds,
	\end{equation*}
	and since $G$ is positive for $t\ge 0$ we denote
	\begin{equation*}
		G(t)_+:=G(t_+), \quad \text{for } t\in\R.
	\end{equation*}
	Since $\Phi$ is monotone and $t\mapsto th_1(t)$ is increasing, then we have 
	\begin{align*}
		&G(t)=1+\int_0^t\sqrt{\Phi(s)h_1(1+s)}ds\nonumber\\
		&\le 1+t\sqrt{\Phi(t)h_1(1+t)}\nonumber\\
		&\le 1+(t+1)\sqrt{\Phi(t)h_1(1+t)}
	\end{align*}
	which leads to
	\begin{equation*}
		G(t)^2\le 8\left(1+\Phi(t)(1+t)^2h_2(1+t)\right).
	\end{equation*}
	Moreover,
	\begin{align}
		\label{stimaDG^2}
		&|D [\eta \, G(|Du|-1)_{+}]|^{2} \nonumber\\
		\leq & \,2\,|D\eta |^{2}[G((|Du|-1)_{+})]^{2}+2\eta ^{2}[G^{\prime
		}((|Du|-1)_{+})]^{2}\,|D((|Du|-1)_{+})|^{2}\nonumber \\
		\leq & 16\,|D\eta |^{2}\, \left [1+\Phi (|Du|-1)_{+}\,h_{2}(1+(|Du|-1)_{+})(1+(|Du|-1)_{+})^{2} \right ]\nonumber\\
		& +2\,\eta
		^{2}\,\Phi (|Du|-1)_{+}h_{1}(1+(|Du|-1)_{+})\,|D(|Du|-1)_{+}|^{2}.
	\end{align}
	Since $\Phi (|Du|-1)_{+}=0$ on $\{|Du|\le 1\}$, then by \eqref{si usa?} and by \eqref{stimaDG^2} we get
	\begin{align}
		\label{piu}
		&\int_\Omega |D [\eta \, G(|Du|-1)_{+}]|^{2} dx\nonumber\\
		& \le C\int_\Omega (\eta^2+ |D\eta|^2)\left(1+\Phi (|Du|-1)_{+} h_2(1+(|Du|-1)_{+})(1 + (|Du|-1)_+)^2\right)  dx\nonumber\\
		& + C\int_\Omega (\eta^2 +  |D\eta|^2) \left(1+\Phi' (|Du|-1)_{+} h_2(1+(|Du|-1)_{+})(1 + (|Du|-1)_+)^2\right) dx.
	\end{align}
	Furthermore, by Sobolev inequality we have 
	\begin{align}
		\label{meno}
		\left[\int_\Omega|\eta \, G(|Du|-1)_{+}|^{2^*}dx\right]^{\frac{2}{2^*}}\le C_S \int_\Omega|D(\eta \, G(|Du|-1)_{+})|^{2}dx.
	\end{align}
	Let $\Phi(t)=(t+1)^{2\gamma-2}t^2$ with $\gamma\ge 1$. Since $\Phi$ is non-decreasing and $h_1$ is increasing then it follows that for $t\ge 1$
	\begin{align}
		\label{stella}
		G(t)&=1+\int_0^t(1+s)^{\gamma-1}s\sqrt{h_1(1+s)}ds\nonumber\\
		&\ge 1 + \frac{1}{t}\int_0^t(s+1)^{\gamma-1}sds\int_0^t\sqrt{h_1(1+s)}ds\nonumber\\
		&= 1+\frac{1}{t}\left[\frac{(t+1)^{\gamma-1}(t(\gamma-1)+1)-1}{\gamma(\gamma-1)}\right]\int_0^t\sqrt{h_1(1+s)}ds\nonumber\\
		&\ge 1+\frac{1}{t}\left[\frac{(t+1)^{\gamma-1}t(\gamma-1)}{\gamma(\gamma-1)}\right]\int_0^t\sqrt{h_1(1+s)}ds\nonumber\\
		& = 1+\frac{(t+1)^{\gamma-1}}{\gamma}\int_0^t\sqrt{h_1(1+s)}ds\nonumber\\
		& \ge 1+\frac{t^{\gamma-1}}{\gamma}\int_0^t\sqrt{h_1(1+s)}ds
	\end{align}
	where in first inequality we used \cite[Lemma 3.4-(v)]{Marc93}. Since $h_1$ is increasing and not identically zero, then we can assume $h_1(t)>0$ for $t>1$. Moreover, let $c_3:=\int_0^1 \sqrt{h_1(1+s)}ds$ (we recall that $c_1$ and $c_2$ were already defined in \eqref{H2} and \eqref{H3} respectively), then for every $t\ge 1$ we have
	\begin{align*}
		2\int_0^t \sqrt{h_1(1+s)}ds&\ge c_3+\int_0^t \sqrt{h_1(1+s)}ds\\
		&\ge \min\{c_3,1\}\left(1+\int_0^t \sqrt{h_1(1+s)}ds\right).
	\end{align*}
	By \eqref{H2} and the previous inequality, we get
	\begin{equation*}
		\left(\int_0^t \sqrt{h_1(1+s)}ds\right)^\alpha\ge c_4 h_2(t)t^2, \quad t\ge 1,
	\end{equation*}
	for some positive $c_4$ depending on $c_1$ and $c_3$. By hypothesis \eqref{H2} it is $\alpha\in [2,2^+)$, so for $t\ge 1$ we have
	\begin{align*}
		\left(\int_0^t \sqrt{h_1(1+s)}ds\right)^{2^*}&\ge \left(\int_0^t \sqrt{h_1(1+s)}ds\right)^{\alpha}\left(h_1(1)(t-1)+c_3\right)^{2^*-\alpha}\nonumber\\
		& \ge{\min\{c_5,c_3\}^{2^*-\alpha}c_4h_2(t)t^{2^*-\alpha+2}},
	\end{align*}
	\color{black}
	with $c_5:=h_1(1).$\\
	From the previous inequality and \eqref{stella} we get for $t\ge 1$
	\begin{align*}
		G(t)^{2^*}&\ge \frac{1}{2}\left[1+\left(\frac{t^{\gamma-1}}{\gamma}\int_0^t\sqrt{h_1(1+s)}ds\right)^{2^*}\right]\nonumber\\
		& \ge \frac{1}{2} + \min\{c_5,c_3\}^{2^*-\alpha}\frac{c_4}{2\gamma^{2^*}}t^{2^*\gamma-\alpha+2}h_2(t)
	\end{align*}
	\color{black}
	which means that there exists a constant $c_6$ depending on $c_1,c_3,c_4,c_5,$ and $\alpha$ such that for $t\ge 1$
	\begin{align}
		\label{disugG}
		c_6\gamma^{2^*}G(t)^{2^*}&\ge 1+t^{2^*\gamma-\alpha+2}h_2(t).
	\end{align}
	Since $\Phi(t)=(t+1)^{2\gamma-2}t^2$, then $\Phi'(t)=2(1+t)^{2\gamma-3}t(1+\gamma)$. So,
	\begin{equation*}
		\Phi(t)\le (t+1)^{2\gamma}, \quad \Phi'(t)\le 2(1+\gamma) (1+t)^{2\gamma}.
	\end{equation*}

	Plugging these functions in \eqref{piu} we get
	\begin{align*}
		& \int_{\Omega }|D(\eta \,G(|Du|-1)_{+}|^{2}\,dx \\
		\le & 2(\gamma+1) C \int_\Omega (\eta^2+ |D\eta|^2)\left(1+ (1+(|Du|-1)_{+})^{2(\gamma+1)} h_2(1+(|Du|-1)_{+})\right)  dx.
	\end{align*}
	From the previous estimate, \eqref{meno}, and \eqref{disugG} we obtain
	\begin{align*}
		&\left( \int_\Omega\eta^2(1+(1+(|Du|-1)_+))^{2^*\gamma-\alpha+2}h_2(1+(|Du|-1)_+) dx  \right)^{\frac{2}{2^*}}\\
		& \le C_6(\gamma^4(\gamma+1)) \int_\Omega (\eta^2+ |D\eta|^2)\left(1+ (1+(|Du|-1)_{+})^{2(\gamma+1)} h_2(1+(|Du|-1)_{+})\right)  dx,\\
		& \le C_6(\gamma+1)^5\int_\Omega (\eta^2+ |D\eta|^2)\left(1+ (1+(|Du|-1)_{+})^{2(\gamma+1)} h_2(1+(|Du|-1)_{+})\right)  dx,
	\end{align*}
	with $C_6$ depending on $C,c_1,c_3,c_4,c_5,c_6,C_S$, and $\alpha$. By the properties of $\eta$, we get
	\begin{align}
		\label{asterisco2}
		&\left( \int_{B_\rho}1+(1+(|Du|-1)_+)^{2^*\gamma-\alpha+2}h_2(1+(|Du|-1)_+) dx  \right)^{\frac{2}{2^*}}\nonumber\\
		&\le \frac{C_6 2(\gamma+1)^5}{(R-\rho)^2} \int_{B_R} 1+ (1+(|Du|-1)_{+})^{2(\gamma+1)} h_2(1+(|Du|-1)_{+})  dx.
	\end{align}
	Now we apply the same iterative argument as in \cite{ EMMP, EPT, Marc96}. Fix $\bar \rho$ and $\bar R$ such that $\bar \rho<\bar R$ and let $\rho_i:=\bar \rho +\frac{\bar R - \bar \rho}{2^i}$ for $i\in \N_0$ so that $\rho_0=\bar R$ and $\lim_{i\to \infty}\rho_i=\bar \rho$. Let
	\begin{equation*}
		\delta_0=2, \quad \delta_i=\frac{2^*}{2}\delta_{i-1} -\frac{\alpha}{2},
	\end{equation*}
	so that
	\begin{equation*}
		2^*\delta_i-\alpha+2=2(\delta_{i+1}+1),
	\end{equation*}
	which implies that $\delta_i$ is increasing for $i\in\N$, $\lim_{i\to \infty}\delta_i=+\infty$, and $\delta_i\ge 1$ for $i$ large enough. Fixed $i\in \N$, we choose $\gamma=\delta_i$, $R=\rho_i$, and $\rho=\rho_{i+1}$ in \eqref{asterisco2} and we get
	\begin{align*}
		&\left( \int_{B_{\rho_{i+1}}}1+(1+(|Du|-1)_+)^{2^*\delta_{i}-\alpha+2}h_2(1+(|Du|-1)_+) dx  \right)^{\frac{2}{2^*}}\nonumber\\
		&\le \frac{C_6 2^i(\gamma+1)^5}{(R-\rho)^2} \int_{B_{\rho_{i}}} 1+ (1+(|Du|-1)_{+})^{2(\delta_i+1)} h_2(1+(|Du|-1)_{+})  dx.\\
		&= \frac{C_6 2^i(\gamma+1)^5}{(R-\rho)^2} \int_{B_{\rho_{i}}} 1+ (1+(|Du|-1)_{+})^{2^*\delta_{i-1}-\alpha+2} h_2(1+(|Du|-1)_{+})  dx.
	\end{align*}
	Iterating the inequality another $i-1$ times we get
	\begin{align*}
		&\left( \int_{B_{\rho_{i+1}}}1+(1+(|Du|-1)_+)^{2(\delta_{i}+1)}h_2(1+(|Du|-1)_+) dx  \right)^{\left(\frac{2}{2^*}\right)^i}\nonumber\\
		&\le C_7\int_{B_{\rho_{i}}} 1+ (1+(|Du|-1)_{+})^{2} h_2(1+(|Du|-1)_{+})  dx,
	\end{align*}
	with
	\begin{align*}
		C_7=\prod_{k=1}^\infty\left( \frac{(\delta_k+1)^52^kC_6}{(\bar R-\bar \rho)^2}\right)^{(\frac{2}{2^*})^k}.
	\end{align*}
	It is not difficult to prove that
	\begin{equation}
		\label{stimedeltak}\delta_k=2^*\left(\frac{2^*}{2}\right)^k-\frac{\alpha}{2}\left(\frac{\left(\frac{2^*}{2}\right)^k-1}{2^*-2}\right) \le  \left(\frac{2^*}{2}\right)^k\left(\frac{\alpha/2+1}{2^*-2}\right), \quad k\in\N.
	\end{equation}
	Let $c_\alpha:=\frac{\alpha/2+1}{2^*-2}$ and $C_\alpha:=(c_\alpha+1)^5$, then it follows that
	\begin{align*}
		C_7& \le \prod_{k=1}^\infty\left( \frac{\left(\left(\frac{2^*}{2}\right)^kc_\alpha+1\right)^52^kC_6}{(\bar R-\bar \rho)^2}\right)^{(\frac{2}{2^*})^k}\\
		& \le \prod_{k=1}^\infty\left( \frac{\left(\frac{2^*}{2}\right)^{5k}\left(c_\alpha+1\right)^52^kC_6}{(\bar R-\bar \rho)^2}\right)^{(\frac{2}{2^*})^k}\\
		& = \prod_{k=1}^\infty\left( \frac{\left(\frac{2^*}{2}\right)^{5k}C_\alpha 2^kC_6}{(\bar R-\bar \rho)^2}\right)^{(\frac{2}{2^*})^k}\\
		& \le \prod_{k=1}^\infty\left( \frac{(2^*)^{5k}C_\alpha C_6}{(\bar R-\bar \rho)^2}\right)^{(\frac{2}{2^*})^k}\\
		&=\prod_{k=1}^\infty\left(\frac{C_6'}{(\bar R-\bar \rho)^2}\right)^{(\frac{2}{2^*})^k}(2^*)^{5k(\frac{2}{2^*})^k}\\
		&\le\left(\frac{C_6'}{(\bar R-\bar \rho)^2}\right)^{\sum_{k=1}^\infty(\frac{2}{2^*})^k}(2^*)^{\sum_{k=1}^\infty 5k(\frac{2}{2^*})^k},
	\end{align*}
	with $C_6':=C_\alpha C_6$.  Now we observe that
	\begin{equation*}
		\sum_{k=1}^\infty\left(\frac{2}{2^*}\right)^k=\frac{2^*}{2^*-2}, \quad \sum_{k=1}^\infty 5k\left(\frac{2}{2^*}\right)^k=c<\infty,
	\end{equation*}
	which leads to $C_7\le\frac{C_8}{(\bar R-\bar \rho)^{\frac{22^*}{2^*-2}}},$ with $C_8=c(C_6')^{\frac{2^*}{2^*-2}}$, and 
	\begin{align*}
		&\left( \int_{B_{\rho_{i+1}}}1+(1+(|Du|-1)_+)^{2(\delta_{i+1}+1)}h_2(1+(|Du|-1)_+) dx  \right)^{(\frac{2}{2^*})^i}\nonumber\\
		&\le \frac{C_8}{(\bar R-\bar \rho)^{\frac{22^*}{2^*-2}}}\int_{B_{\bar R}} 1+ (1+(|Du|-1)_{+})^{2} h_2(1+(|Du|-1)_{+})  dx.
	\end{align*}
	By the properties of $h_2$, we have that for $0\le r\le s$ and for every $t\ge 1$ it holds
	\begin{align*}
		t^r&\le 1+t^sh_2(t)\le \frac{t^sh_2(1)}{h_2(1)}+t^sh_2(t)\\
		&\le \frac{t^sh_2(t)}{h_2(1)}+t^sh_2(t)=\left(1+\frac{1}{h_2(1)}\right)t^sh_2(t),
	\end{align*}
	which leads to
	\begin{align*}
		&\left( \int_{B_{\rho_{i+1}}}(1+(|Du|-1)_+)^{2(\delta_{i+1}+1)} dx  \right)^{(\frac{2}{2^*})^i}\nonumber\\
		&\le \frac{C_9}{(\bar R-\bar \rho)^{\frac{22^*}{2^*-2}}}\int_{B_{\bar R}}(1+(|Du|-1)_{+})^{2} h_2(1+(|Du|-1)_{+})  dx.
	\end{align*}
	for some positive constant $C_9$ proportional to $C_8$.
	Now we observe that by \eqref{stimedeltak} we have
	\begin{equation*}
		\lim_{i\to \infty}2(\delta_{i+1}+1)\left(\frac{2}{2^*}\right)^i=22^*-\frac{\alpha}{2^*-2}.
	\end{equation*}
	So, passing to the limit for $i\to \infty$ we get
	\begin{align}
		\label{apriori1}
		&||1+(|Du|-1)_+||_{L^\infty(B_{\bar \rho})}^{22^*-\frac{\alpha}{2^*-2}}\nonumber\\
		& \le\limsup_{i\to \infty} \left(\int_{B_{\rho_{i+1}}}(1+(|Du|-1)_{+})^{2(\delta_{i+1}+1)} h_2(1+(|Du|-1)_{+})  dx\right)^{(\frac{2}{2^*})^i}\nonumber\\
		& \le \frac{C_9}{(\bar R-\bar \rho)^{\frac{22^*}{2^*-2}}}\int_{B_{\bar R}}(1+(|Du|-1)_{+})^{2} h_2(1+(|Du|-1)_{+})  dx.
	\end{align}
	Let now $H:[0,+\infty)\to \R$ be defined as
	\begin{equation*}
		H(t):=1+\int_0^t\frac{s}{s+1}\sqrt{h_1(1+s)}ds.
	\end{equation*}
	Since $h_1$ is increasing, then for $t\ge 1$
	\begin{align*}
		H(t)^2&\le \left(1+t\sqrt{h_1(1+t)}\right)^2\\
		&\le 2 \left(1+t^2h_1(1+t)\right)\le 2 \left(1+t^2h_2(1+t)\right)\\
		&\le 4 (1+t)^2h_2(1+t).
	\end{align*}
	Moreover,
	\begin{align}
		\label{stimaDH^2}
		&|D [\eta \, H(|Du|-1)_{+}]|^{2} \nonumber\\
		\leq & \,2\,|D\eta |^{2}[H((|Du|-1)_{+})]^{2}+2\eta ^{2}[H^{\prime
		}((|Du|-1)_{+})]^{2}\,|D((|Du|-1)_{+})|^{2}\nonumber \\
		\leq & 16\,|D\eta |^{2}\, \left [h_{2}(1+(|Du|-1)_{+})(1+(|Du|-1)_{+})^{2} \right ]\nonumber\\
		& +2\,\eta
		^{2}\,h_{1}(1+(|Du|-1)_{+})\,|D(|Du|-1)_{+}|^{2}.
	\end{align}
	Repeat now argument that leads to \eqref{si usa?2} but with $\Phi(t):=\frac{t}{t+1}$. Since $\Phi(t)+(1+t)\Phi'(t)=1$ for every $t\in\R$, we get
	\begin{align}
		\label{è giusta?}
		&\int_\Omega\eta^2  h_1(1+(|Du|-1)_{+})|D(|Du|-1)_+|^2dx\nonumber\\
		& \le C\int_\Omega (\eta^2+ |D\eta|^2) h_2(1+(|Du|-1)_{+})(1 + (|Du|-1)_+)^2 dx.
	\end{align}
	Putting together \eqref{è giusta?} and \eqref{stimaDH^2} we get
	\begin{equation*}
		\int_\Omega |D [\eta \, H(|Du|-1)_{+}]|^{2}dx\le C\int_\Omega (\eta^2+ |D\eta|^2) h_2(1+(|Du|-1)_{+})(1 + (|Du|-1)_+)^2 dx.
	\end{equation*}
	By Sobolev inequality, we get
	\begin{align}
		\label{SobvolH}
		&\left[\int_\Omega | \eta \, H(|Du|-1)_{+}|^{2^*}dx\right]^{\frac{2}{2^*}}\nonumber\\
		&\le C_S\int_\Omega (\eta^2+ |D\eta|^2) h_2(1+(|Du|-1)_{+})(1 + (|Du|-1)_+)^2 dx.
	\end{align}
	Let $\delta:=\frac{2^*}{\alpha}$, then by hypothesis \eqref{H2} we have 
	\begin{equation*}
		H(t)^{\frac{2^*}{\delta}}\ge \frac{1}{c_1}t^2h_2(t), \quad t\ge 1,
	\end{equation*}
	so \eqref{SobvolH} becomes
	\begin{align*}
		&\left[\int_\Omega \eta^{2^*}\left( h_2(1+(|Du|-1)_{+})(1 + (|Du|-1)_+)^2\right)^{\delta} dx\right]^{\frac{2}{2^*}}\\
		&\le c_2'\int_\Omega (\eta^2+ |D\eta|^2) h_2(1+(|Du|-1)_{+})(1 + (|Du|-1)_+)^2 dx.
	\end{align*}
	with $c_2'$ depending on $c_1$ and $C_S$. If we set
	\begin{equation*}
		V(x):=(1+(|Du|-1)_{+})^{2} h_2(1+(|Du|-1)_{+}),
	\end{equation*}
	then by the properties of $\eta$ we have
	\begin{align*}
		&\left[\int_{B_{ \rho}} \left( h_2(1+(|Du|-1)_{+})(1 + (|Du|-1)_+)^2\right)^{\delta} dx\right]^{\frac{2}{2^*}}\\
		&\le \frac{4c_2'}{(R-\rho)^2}\int_{B_{ R}} h_2(1+(|Du|-1)_{+})(1 + (|Du|-1)_+)^2 dx,
	\end{align*}
	i.e.
	\begin{equation*}
		\left[\int_{B_{ \rho}} V(x)^{\delta} dx\right]^{\frac{2}{2^*}}\le \frac{4c_2'}{(R-\rho)^2}\int_{B_{ R}}V(x)dx.
	\end{equation*}
	
	Let $\gamma\ge\frac{2^*}{2}$, then by H\"older inequality we obtain
	\begin{align*}
		\int_{B_\rho}V(x)^\delta  dx\le\frac{4c_2'}{(R-\rho)^2}\left(\int_{B_R}V(x)^{\delta}  dx\right)^{\frac{2^*}{2\gamma}}\left(\int_{B_R}V(x)^{\frac{\gamma-\delta}{\gamma-1}}  dx\right)^{\frac{\gamma-1}{\gamma}}.
	\end{align*}
	As before, we consider $\rho_i:= \bar R - \frac{\bar R-\bar\rho}{2^i}$ with $i\in\N_0$, and we let $\rho=\rho_i$ and $R=\rho_{i+1}$. Fixed $i\in\N$, we iterate the previous inequality $i$-times and we pass to the limit for $i\to \infty$ obtaining
	\begin{align*}
		&\int_{B_{\bar \rho}}V(x)^\delta  dx\le \frac{c_3'}{(\bar R-\bar \rho)^{\frac{22^*\gamma}{2\gamma-2^*}}}\left(\int_{B_{\bar R}}V(x)^{\frac{\gamma-\delta}{\gamma-1}}  dx\right)^{2^*\frac{\gamma-1}{2\gamma-2^*}}.
	\end{align*}
	Let $\gamma$ be such that $\frac{\gamma-\delta}{\gamma-1}=\frac{1}{\beta}$, i.e.
	\begin{equation*}
		\gamma=\frac{\delta\beta-1}{\beta-1},
	\end{equation*}
	with $\beta$ defined in \eqref{H3}, then using \eqref{H3} in the previous inequality we get
	\begin{equation*}
		\int_{B_{\bar \rho}}V(x)^\delta  dx\le c_4' \left(\int_{B_{\bar R}}1+f(Du)  dx\right)^{2^*\frac{\gamma-1}{2\gamma-2^*}},
	\end{equation*}
	with $c_4'$ depending on $c_2,c_3',\alpha,\beta, \bar \rho, \bar R$ and $n$. Using H\"older inequality again we obtain
	\begin{align*}
		&\int_{B_{\bar \rho}}V(x)  dx\le |B_{\bar\rho}|^{1-\frac{1}{\delta}}\left(\int_{B_{\bar \rho}}V(x)^\delta  dx\right)^{\frac{1}{\delta}}\le c_5'\left(\int_{B_{\bar R}}1+f(Du)   dx\right)^{2^*\frac{\gamma-1}{(2\gamma-2^*)\delta}}.
	\end{align*}
	By \eqref{apriori1}, we get
	\begin{align*}
		||V||_{L^\infty(B_{\bar \rho},\R^3)}\le \left[c_6'\left(\int_{B_{\bar R}}1+f(Du)   dx\right)^{2^*\frac{\gamma-1}{(2\gamma-2^*)\delta}}\right]^{\frac{2^*-2}{2^*-\alpha}}.
	\end{align*}
	To conclude, we observe that $V\ge h_2(1)|Du|$ which leads to
	\begin{align*}
		||Du||_{L^\infty(B_{\bar \rho},\R^3)}\le c_7'\left(\int_{B_{\bar R}}1+f(Du)   dx\right)^{\theta},
	\end{align*}
	with $\theta=2^*\frac{\gamma-1}{(2\gamma-2^*)\delta}\frac{2^*-2}{2^*-\alpha}$, and this proves \eqref{apriori}.
	
	\section{Proof of Theorem \ref{Lipschitz}}
	\noindent
	The proof of the theorem is divided in three steps. In the first step we consider a suitable sequence of approximating functionals. In the second step we consider the minimizers of the aformentioned functionals with appropriate boundary conditions such that they satisfy the BSC (see \ref{bscdefn}). Finally, in step three, we pass to the limit exploiting the coercivity of $f$.\\
	
	\noindent
	{\it Step 1} By the same argument as in  \cite[Theorem 2.1]{EMMP}, there exists  a sequence $f_k\in\mathcal{C}^2(\R^n)$ of locally uniformly convex functions such that 
	\begin{enumerate}
		\item $\tilde f_k$ satisfies \eqref{H1}-\eqref{H3} with $2h_2$ instead of $h_2$,  constants 
		$C_{1}$ and $C_{2}$ independent from $k$;
		\item $\tilde f_k$ uniformly converges to $f$ on compact sets;
		\item for every $\delta >0$ and for every $%
		k$ sufficiently large 
		\begin{equation}
			\label{delta}
			f(\xi)\leq
			\begin{cases}
				\tilde f_k(\xi)+\delta & \text{ if $|\xi |\leq t_0+2$}\\
				\tilde f_k(\xi)        & \text{ if $|\xi |> t_0+2$};
			\end{cases}
		\end{equation}
		\item for every  $\xi\in\R^n$
		$$
		f(\xi)-1\le \tilde f_k(\xi)\le f(\xi)+|\xi|+1.
		$$
	\end{enumerate}
	Indeed, for every $k \in \mathbb{N}$, consider the sequence $f_k$ defined as follows (see \cite{MS})
	\[
	f_k(\xi) = f(\xi)(1 - \phi(\xi)) + (f \phi) * \eta_k(\xi),
	\]
	where $(\eta_k)_k$ is a family of standard mollifiers and $\phi \in \mathcal C^\infty(\mathbb{R}^n)$, $0 \leq \phi(\xi) \leq 1$ for every $\xi \in \mathbb{R}^n$, $\phi(\xi) = 1$ if $|\xi| \leq t_0 + 1$ and $\phi(\xi) = 0$ if $|\xi| \geq t_0 + 2$. By construction, $f_k \in C^2(\mathbb{R}^n)$ and
	\begin{equation*}
		f_k(\xi)=
		\begin{cases}
			f(\xi) \quad |\xi|\ge t_0+2\\
			(f\phi)\ast\eta_k(\xi) \quad |\xi|\le t_0+1
		\end{cases}
	\end{equation*}
	so it follows that $f_k \to f$ uniformly. We can assume that $|f(\xi) - f_k(\xi)| \leq 1$ for every $\xi \in \mathbb{R}^n$ and $k \in \mathbb{N}$. Furthermore, it is possible to prove that for sufficiently large $k$, $f_k$ is a convex function and $D^2 f_k(\xi)$ is positive definite for $|\xi| > t_0 + 1$. Since $f_k(\xi) = f(\xi)$ for $|\xi| > t_0 + 2$, then \eqref{H1}, \eqref{H2}, and \eqref{H3} hold with $t_0 + 2$ instead of $t_0$.
	Let $h \in \mathcal C^2([0,+\infty))$ be the positive, increasing, convex function defined by
	\[
	h(t) = \begin{cases}
		\frac{1}{8}(6t^2-t^4+3), & \text{if } t \in [0,1), \\
		t, & \text{if } t \geq 1.
	\end{cases}
	\]
	Observe that, $h''$ is nonnegative and $h'' > 0$ in $[0,1)$. For $k \in \mathbb{N}$ denote
	\[
	\tilde{f}_k(\xi) = f_k(\xi) + \frac{1}{k} h\left(\frac{|\xi|}{t_0+2}\right), \quad \text{for every } \xi \in \mathbb{R}^n,
	\]
	and notice that $\tilde{f}_k \in C^2(\mathbb{R}^n)$, $|D^2\tilde f_k|\ge C$, for some $C\in\R$, and that it is uniformly convex on compact subsets of $\mathbb{R}^n$. 
	Moreover, by \eqref{H1} and \cite[(3.3)]{MP}
	\begin{equation*}
		h_1(|\xi|)|\lambda|^2\le\sum_{i,j=1}^n (\tilde f_k)_{\xi_i, \xi_j}(\xi)\lambda_i\lambda_j\le \left(h_2(|\xi|)+\frac{1}{k(t_0+2)}\frac{1}{|\xi|}\right)|\lambda|^2
	\end{equation*}
	for every $\xi, \eta \in \mathbb{R}^n$, $|\xi| \geq t_0 + 2$. On the other hand, since $h_2$ is increasing
	\[
	t h_2(t) \geq (t_0 + 2)h_2(t_0 + 2) \quad \text{for } t \geq t_0 + 2.
	\]
	So for $k$ sufficiently large, we get that
	\begin{equation*}
		h_2(|\xi|)+\frac{1}{k(t_0+2)}\frac{1}{|\xi|}\le h_2(|\xi|)+\frac{h_2(t)}{k(t_0+2)h_2(t_0+2)}\frac{1}{|\xi|} \le 2h_2(t)
	\end{equation*}
	Therefore, $\tilde{f}_k$ satisfies \eqref{H1}, \eqref{H2}, and \eqref{H3} with $t_0 + 2$ instead of $t_0$, $2h_2$ instead of $h_2$, and constants $C_1$ and $C_2$ independent from $k$.
	
	Now, define the integral functional
	\begin{equation}
		\label{approx. functional}
		F_k(v) = \int_{B_R} \tilde{f}_k(Dv) +g(x, v) \, dx, \quad v \in W^{1,1}(B_R),
	\end{equation}
	where $B_R$ is a ball compactly contained in $\Omega$. Let $u \in W^{1,1}_{\text{loc}}(\Omega)$ be a local minimizer of the functional \eqref{functional} and let $u_\varepsilon$ be a mollification of $u$. Since by construction $u_\varepsilon \in C^2(\overline{B_R})$, then $u_\varepsilon$ verifies the bounded slope condition (see for instance \cite{Miranda} and \cite[Theorem 1.1 and Theorem 1.2]{Giusti}). By construction, $\tilde{f}_k$ is superlinear and $F_k$ has a unique minimizer $v_{\varepsilon,k}$ among Lipschitz continuous functions in $B_R$ with boundary value $u_\varepsilon$ on $\partial B_R$.\\
	
	\noindent
	{\it Step 2} In this step we prove that $v_{k.\varepsilon}\in W^{1,\infty}(B_R)$ when $R$ is sufficiently small. 
	We recall that if $u_{\varepsilon}$ satisfies the (BSC) on $B_R$, then for every $z \in \partial B_R,$ there exists $\kappa_z^-$ and $\kappa_z^+$ such that
	\begin{equation}\label{BSC}
		\kappa_z^-\cdot (x - z) + u_{\varepsilon}(z) \le \, u_{\varepsilon}(x) \le \kappa_z^+ \cdot(x - z) + u_{\varepsilon}(z) \qquad \forall x \in \partial \Omega,
	\end{equation}
	see \cite{Miranda, Giusti} for more details.  Our aim is to construct lower and upper Lipschitz barriers for the boundary datum. We use the same argument as in \cite{EPT, FT, GT}.\\
	
	\noindent
	First, we recall that by Proposition \cite[Proposition 2.1-ii)]{EPT}  we have 
	that $\tilde f_k^*$  is defined in $\R^n$ and superlinear. Now, let us fix $z\in \partial B_R$ and let $\kappa_z^-$ as in the \eqref{BSC}. Consider now the set
	\[
	\left \{\frac{n}{L} \tilde f_k^*\left (\frac{L}{n}x \right) - \kappa_z^- \cdot x - c \le 0\right \} =  \Omega_{\kappa_z^-, c}.
	\]
	For $c$ sufficiently large $\Omega_{\kappa_z^-, c}$ is not empty and convex. Moreover, since $\tilde f^*_k$ is superlinear, then $\Omega_{\kappa_z^-, c}$ is bounded. Furthermore, the fact that $\tilde f^*_k$ is finite for every $x\in\R^n$ yields that 
	\begin{equation*}
		\lim_{c\to \infty} \min\{|x|:\ x\in\partial \Omega_{\kappa_z^-, c}\}=+\infty.
	\end{equation*}
	
	By Proposition \ref{polare} it  follows that, for $c$ sufficiently large,  $\partial \Omega_{\kappa_z^-, c}$ is $\mathcal{C}^2$. Moreover, there exists $s_0\in \R$ such that $\tilde f_k^*\in\mathcal{C}^2(\R^n\setminus B_{s_0}(0))$, and $|D^2\tilde f_k|\ge C$. Observe now that if we fix $\frac{L}{n}x\in\R^n\setminus B_{s_0}$ and we define
	\[
	\varphi(t)=\tilde f_k^*\left(t\frac{x}{|x|}\right)\quad\text{ for }\, t\ge s_{0}.
	\]
	It follows that, for $t=\frac{L}{n}|x|$, there exists a non negative constant $C_0$ such that
	\[
	\left|Df_k^*\left(\frac{L}{n}x\right)\right|\ge \varphi'\left(\frac{L}{n}|x|\right)\ge \int_{s_0}^{\frac{L}{n}|x|}C\,d\tau +C_0,
	\]
	and the last term goes to $+\infty$ as $|x|\to +\infty$. We can now do the same computations as in Step 2 of the proof of \cite[Theorem 4.5]{GT}. It follows that the principal curvatures of $\partial \Omega_{\kappa_z^-, c}$ at every point $x$ are less or equal to
	
	\begin{equation*}
		\frac{|D^2 \tilde f^*_k(\frac{L}{n} x)|}{|D \tilde f^*_k
			(\frac{L}{n} x)|} \le \frac{1}{CC_0}.
	\end{equation*}
	Let $\delta_0=CC_0$. Let $R<\delta_0$ and $\nu$ be the normal vector to $\partial B_R$ in $z.$ Let $x_z \in \partial \Omega_{\kappa_z^-, c}$ be such that its normal vector is exactly $\nu$. 
	
	Consider now the function
	\[
	v_z(x):=\frac{n}{L} \tilde f^*_k\left (\frac{L}{n} (x - (z - x_z)) \right) + u_{\varepsilon}(z) - \frac{n}{L} \tilde f^*_k\left (\frac{L}{n} x_z \right), 
	\]
	and define the set
	\begin{equation*}
		{\tilde\Omega_{\kappa_z^-,c}}=\left\{v_z(x) -\kappa^-_z\cdot(x-(z-x_z))- u_{\varepsilon}(z) + \frac{n}{L} \tilde f^*_k\left (\frac{L}{n} x_z \right)-c\le 0\right\}.
	\end{equation*}
	Clearly, ${\tilde\Omega_{\kappa_z^-,c}}= \Omega_{\kappa_z^-,c} +(z-x_z)$ so the curvature of $\partial{\tilde\Omega_{\kappa_z^-,c}}$ in $z$ is the same of $\partial\Omega_{\kappa_z^-,c}$ in $x_z$. Since $R<\delta_0$ we have that $B_R\subset {\tilde\Omega_{\kappa_z^-,c}}$ and $z\in \partial B_R\cap \partial{\tilde\Omega_{\kappa_z^-,c}}$. Moreover, we have that
	\begin{equation*}
		{\tilde\Omega_{\kappa_z^-,c}} =\{v_z(x) \le \kappa^-_z\cdot(x-z)+ u_{\varepsilon}(z)\}.  
	\end{equation*}
	By applying the comparison principle (in \cite[Theorem 2.4]{FT} and in \cite[Theorem 2.4]{GT}) between the minimizer $v_{k,\varepsilon}$ and the function $v_k$ we get that $v_{k,\varepsilon}(x)\ge v_z(x)$ a.e. in $B_R$. 
	To obtain the lower barrier it is enough to consider for every $x\in B_R$ the function
	\begin{equation*}
		\ell^-(x):= \sup_{z\in\partial B_R} v_z(x).
	\end{equation*}
	Repeating an analogous construction we can construct also the upper barrier $\ell^+$.
	
	By construction, the functions $\ell^\pm$ are Lipschitz in $B_R$. It follows that by \eqref{G4},\eqref{G3}, and \cite[Theorem 2.2]{GT} (see also \cite[Theorem 5.2]{MTrado}) we have $v_{k.\varepsilon}\in W^{1,\infty}(B_R)$.\\
	
	\noindent
	\textit{Step 3}
	Now we apply Lemma \ref{Lemmabound} with $f_-(\xi)=f(\xi)-1$ and $f_+(\xi)= f(\xi)+|\xi|+1$. It follows that there exists a constant $\tilde M$ such that $\|v_{k,\varepsilon}\|_{L^\infty(B_R)}\le \tilde M$ for every $k$ and $\varepsilon$.
	By hypothesis \eqref{G1} and \eqref{G2} we get that there exists a constant $\tilde K$ such that that $\|g(x,u_{\varepsilon})\|_{L^{1}(B_R)} \le \tilde K$ and $\|g(x,v_{k,\varepsilon})\|_{L^{1}(B_R)} \le \tilde K$.
	From Step 2 we get that $v_{k,\varepsilon}\in W^{1,\infty}(B_R)$
	and from Lemma \ref{apriori_estimate} we have
	\[
	\|D v_{k,\varepsilon}\|_{L^\infty(B_\rho; \mathbb{R}^n)}\le C \, \left(\frac{1}{(R-\rho)^n}\int_{B_R} 1+\tilde f_k(D v_{k,\varepsilon})\,dx\right)^\theta
	\]
	with $C$ not depending on $k$ and $\varepsilon$. Now we add and subtract $g(x,v_{k,\varepsilon})$ in the right hand side and from the minimality of $v_{k,\varepsilon}$ we get
	\[
	\begin{split}
		\|Dv_{k,\varepsilon}\|_{L^\infty(B_\rho; \mathbb{R}^n)}
		\le & C \left(\frac{1}{(R-\rho)^n}\int_{B_R} 1+f_k(D v_{k,\varepsilon})+g(x,v_{k,\varepsilon})\,dx \right.\\
		& \left. -\frac{1}{(R-\rho)^n}\int_{B_R}  g(x,v_{k,\varepsilon})\,dx\right)^\theta \\
		\le & C \, \left(\frac{1}{(R-\rho)^n}\int_{B_R} 1+f_k(D u_{\varepsilon})+g(x,u_{\varepsilon})\,dx +\frac{\tilde K}{(R-\rho)^n}\right)^\theta .
	\end{split}
	\]
	It follows that
	\[
	\limsup_{k\to\infty}\|D v_{k,\varepsilon}\|_{L^\infty(B_\rho; \mathbb{R}^n)}
	\le M_\varepsilon
	\]
	with 
	\[
	M_\varepsilon= C
	\left[\frac{1}{(R-\rho)^n}\left(\int_{B_R} 1+f(D u_{\varepsilon})+g(x,u_{\varepsilon})\,dx +\tilde K\right)\right]^\theta .
	\]
	Since $(v_{\varepsilon ,k})_k$ is uniformly bounded in $W^{1,\infty }(B_{\rho })$, then there exists a subsequence $k_{j}\rightarrow \infty$
	such that $(v_{\varepsilon ,k_{j}})_j$ is weakly$^{\ast }$ convergent in $W^{1,\infty }(B_{\rho })$. 
	Now we fix a sequence $\rho_j\rightarrow R$ and, 
	by a diagonalization argument, we extract a subsequence, that we still denote by $(v_{\varepsilon ,k_{j}})_j$, weakly$^*$ converging to
	$\bar{v}_{\varepsilon }$ in $W^{1,\infty }(B_{\rho })$ for every $\rho<R$.
	
	We recall that $(v_{\varepsilon ,k_{j}})\subset
	u_\varepsilon+W_{0}^{1,1}(B_{R})$ and for every $\rho <R$
	\begin{equation}
		\Vert D\bar{v}_{\varepsilon }\Vert _{L^{\infty }(B_{\rho}; \mathbb{R}^n)}\leq
		M_\varepsilon.  \label{normepsilon}
	\end{equation}
	
	Now we want to prove that $v_{\varepsilon ,k_{j}}$ weakly converges to $\bar{v}_{\varepsilon }$
	in $W^{1,1}(B_{R})$ so that $\bar{v}_{\varepsilon }\in u_\varepsilon+W_{0}^{1,1}(B_{R})$. Indeed, by the minimality of $v_{\varepsilon ,k_{j}}$, as $j\rightarrow \infty $ we have 
	\begin{equation*}
		\begin{split}
			\int_{B_{R}}f(Dv_{\varepsilon ,k_{j}})\,dx
			& \leq \int_{B_{R}}1+\tilde f_{k_{j}}(Dv_{\varepsilon ,k_{j}})
			+g(x,v_{\varepsilon ,k_{j}})\,dx - \int_{B_{R}}g(x,v_{\varepsilon ,k_{j}})\,dx \\
			& \leq \int_{B_{R}}{\tilde f}_{k_{j}}(Du_{\varepsilon})
			+g(x,u_{\varepsilon})\,dx+(\tilde K +1)  \\
			&\rightarrow \int_{B_{R}}f(Du_{\varepsilon
			})+g(x,u_{\varepsilon})\,dx+(\tilde K+1).
		\end{split}
	\end{equation*}%
	The superlinearity of $f$ and  de la Vall\'{e}e-Poussin Theorem imply that we can choose the sequence $k_{j}$ such
	that $Dv_{\varepsilon ,k_{j}}\rightharpoonup D\bar{v}_{\varepsilon }$ in $%
	L^{1}(B_{R}; \mathbb{R}^n)$ and then $(v_{\varepsilon ,k_{j}}-u_{\varepsilon
	})\rightharpoonup (\bar{v}_{\varepsilon }-u_{\varepsilon })\in
	W_{0}^{1,1}(B_{R})$.
	On the other hand, since $v_{\varepsilon,k_{j}}$ is a minimizer of  \eqref{approx. functional}, then
	\begin{equation*}
		\begin{split}
			\int_{B_{R }}&f(Dv_{\varepsilon ,k_{j}}) + g(x,v_{\varepsilon ,k_{j}})\,dx\\ 
			= &\int_{B_{R}}{\tilde f}_{k_{j}}(Dv_{\varepsilon ,k_{j}})+g(x,v_{\varepsilon ,k_{j}})\,dx+\int_{B_{R}}(f(Dv_{\varepsilon ,k_{j}})-{\tilde f}_{k_{j}}(Dv_{\varepsilon ,k_{j}})) \, dx \\
			\leq & \int_{B_{R}} \tilde f_{k_{j}}(Du_{\varepsilon })+g(x,u_{\varepsilon})\,dx+\int_{B_{R}}(f(Dv_{\varepsilon ,k_{j}})-{\tilde f}_{k_{j}}(Dv_{\varepsilon ,k_{j}})) \, dx.
		\end{split}%
	\end{equation*}%
	By \eqref{delta}, for every $\delta>0$ there exists $\bar k$ such that for every $k_j>\bar k$ 
	\begin{equation*}
		\int_{B_{R }}f(Dv_{\varepsilon ,k_{j}}) + g(x,v_{\varepsilon ,k_{j}})\,dx\leq \int_{B_{R}} \tilde f_{k_{j}}(Du_{\varepsilon })+g(x,u_{\varepsilon})\,dx+\delta |B_{R}|,
	\end{equation*}
	and by lower semicontinuity in $W^{1,1}(B_{R})$, passing to the limit for $
	j\rightarrow \infty$, we get 
	\begin{equation*}
		\begin{split}
			\int_{B_{R }}&f(D\bar{v}_{\varepsilon }) +g(x,\bar v_{\varepsilon })\,dx
			\leq \liminf_{j\rightarrow
				\infty }\int_{B_{R }}f(Dv_{\varepsilon ,k_{j}}) + g(x,v_{\varepsilon ,k_{j}})\,dx\\
			&\leq
			\lim_{j\rightarrow \infty }\int_{B_{R}}\tilde f_{k_{j}}(Du_{\varepsilon })+g(x,u_{\varepsilon})\,dx+\delta |B_{R}| \\
			& =\int_{B_{R}}f(Du_{\varepsilon })+g(x,u_{\varepsilon})\,dx+\delta |B_{R}|
		\end{split}%
	\end{equation*}%
	for every $\delta >0$ and so for  $
	\delta \rightarrow 0$ 
	\begin{equation}\label{Jensen}
		\int_{B_{R}}f(D\bar{v}_{\varepsilon })+g(x,\bar v_{\varepsilon })\,dx\leq \int_{B_{R}}f(Du_{\varepsilon})+g(x,u_{\varepsilon})\,dx.
	\end{equation}
	Since $u_\varepsilon$ is a mollification of $u$, then it follows that, by Jensen inequality and Dominated Convergence theorem,  
	\begin{equation}\label{Jensen2}
		\lim_{\varepsilon\to 0} \int_{B_{R}}f(Du_{\varepsilon})+g(x,u_{\varepsilon})\,dx=\int_{B_{R}}f(Du)+g(x,u)\,dx.
	\end{equation}
	Therefore, the right hand side of \eqref{Jensen} is uniformly bounded w.r.t. $\varepsilon$.
	We can now apply again de la Vall\'{e}e-Poussin Theorem in order to extract  a sequence $\varepsilon _{j}\rightarrow 0$ such that $\bar{%
		v}_{\varepsilon _{j}}-u_{\varepsilon _{j}}\rightharpoonup \bar{v}-u$ in $%
	W_{0}^{1,1}(B_{R})$. By \eqref{Jensen}, \eqref{Jensen2}, and the lower semicontinuity of the functional we get
	\begin{equation*}
		\begin{split}
			& \int_{B_{R}}f(D\bar{v})+g(x,\bar{v})\,dx
			\leq \liminf_{j\rightarrow \infty }\int_{B_{R}}f(D\bar{v}_{\varepsilon _{j}})
			+g(x,\bar v_{\varepsilon _{j}})\,dx\\
			&\leq \lim_{j\rightarrow \infty
			}\int_{B_{R}}f(Du_{\varepsilon _{j}})+g(x,u_{\varepsilon_j})\,dx
			=\int_{B_{R}}f(Du)+g(x,u)\,dx.  
		\end{split}
	\end{equation*}
	From the previous inequality, it follows that $\bar{v}$ is another minimizer for \eqref{functional} with $\Omega
	=B_{R} $. Moreover, by \eqref{normepsilon} 
	we can assume that  $\{\bar{v}_{\varepsilon _{j}}\}_{j}$ is weakly$
	^{\ast }$ convergent to $\bar{v}$ in $W^{1,\infty }(B_{\rho })$ for every $%
	0<\rho <R$.
	Thus, by \eqref{normepsilon} 
	and \eqref{Jensen2},  we get that for every $0<\rho<R$  
	\begin{equation*}
		\begin{split}
			&\Vert D\bar{v}\Vert _{L^{\infty}(B_{\rho }; \mathbb{R}^n)} \leq \,\liminf_{j\rightarrow
				\infty }\Vert D\bar{v}_{\varepsilon _{j}}\Vert _{L^{\infty }(B_{\rho }; \mathbb{R}^n)} \\
			& \leq
			\lim_{j\rightarrow \infty }\,C\,\left\{ \frac{1}{(R-\rho )^{n}}
			\left (\int_{B_{R}}1+f(Du_{\varepsilon _{j}})+g(x,u_{\varepsilon_j})\,dx +\tilde K \right )\right\}^{\theta } \\
			& =\,C\,\left\{ \frac{1}{(R-\rho )^{n}} \left (\int_{B_{R}} f(Du)+g(x,u)\,dx +\kappa \right )\right\}
			^{\theta },
		\end{split}
	\end{equation*}
	where $\kappa = \tilde{K} + |B_R|.$
	
	Since $\bar{v}$ and $u$ are two different minimizers of $F$ in $B_{R}$ and $f(\xi )$ is strictly convex for $|\xi |>t_{0}$ (see \eqref{H1}), by arguing as in 
	\cite[Proof of Theorem 2.1]{EMM} it is possible to prove that the set 
	\begin{equation*}
		E_{0}:=\left\{ x\in B_{R}:\left\vert \frac{Du(x)+D\bar{v}(x)}{2}\right\vert
		>t_{0}\right\}\cap\{Du\neq D\bar{v}\}.
	\end{equation*}%
	has zero measure. Therefore, 
	\begin{align*}
		\Vert Du\Vert _{L^{\infty }(B_{\rho }; \mathbb{R}^n)}&\leq \,\Vert Du+D\bar{v}\Vert
		_{L^{\infty }(B_{\rho }; \mathbb{R}^n)}+\Vert D\bar{v}\Vert _{L^{\infty }(B_{\rho }; \mathbb{R}^n)}\\
		&\leq
		\,2t_{0}+\Vert D\bar{v}\Vert _{L^{\infty }(B_{\rho }; \mathbb{R}^n)}<+\infty,
	\end{align*}
	i.e. $u\in W^{1,\infty}(B_\rho)$.
	
	\color{black}
	\section{Acknowledgments}
	The author have been partially supported through the INdAM-GNAMPA 2025 project ``Minimal surfaces: the Plateau problem and behind'' cup E5324001950001 and the project: Geometric-Analytic Methods for PDEs and Applications (GAMPA) , ref. 2022SLTHCE – cup E53D2300588 0006 - funded by European Union - Next Generation EU within the PRIN 2022 program (D.D. 104 - 02/02/2022 Ministero dell’Università e della Ricerca). This manuscript reflects only the authors’ views and opinions and the Ministry cannot be considered responsible for them.

\end{document}